\documentclass[12pt, reqno]{amsart}

\usepackage{amssymb, amsmath, amsthm}
\usepackage[backref]{hyperref}
\usepackage[alphabetic,backrefs,lite]{amsrefs}
\usepackage{amscd}   
\usepackage{fullpage}
\usepackage[all]{xy} 

\DeclareFontEncoding{OT2}{}{} 

\usepackage{color}

\newtheorem{lemma}{Lemma}[section]
\newtheorem{theorem}[lemma]{Theorem}

\newtheorem{prop}[lemma]{Proposition}
\newtheorem{cor}[lemma]{Corollary}

\newtheorem{claim*}{Claim}
\newtheorem{thm}[lemma]{Theorem}

\newtheorem{problem}[lemma]{Problem}

\theoremstyle{definition}
\newtheorem{remark}[lemma]{Remark}

\newcommand{\Aff}{{\mathbb A}}
\newcommand{\G}{{\mathbb G}}

\newcommand{\PP}{{\mathbb P}}

\newcommand{\F}{{\mathbb F}}
\newcommand{\Q}{{\mathbb Q}}

\newcommand{\Z}{{\mathbb Z}}

\newcommand{\ksep}{{k^{\operatorname{s}}}}

\newcommand{\Xsep}{{X^{\operatorname{s}}}}

\newcommand{\Adeles}{{\mathbb A}}
\newcommand{\kk}{{\mathbf k}}

\newcommand{\pp}{{\mathfrak p}}

\newcommand{\Ptilde}{\widetilde{P}}

\newcommand{\calA}{{\mathcal A}}
\newcommand{\calB}{{\mathcal B}}

\newcommand{\calE}{{\mathcal E}}

\newcommand{\calO}{{\mathcal O}}

\newcommand{\OO}{{\mathcal O}}

\DeclareMathOperator{\HH}{H}

\DeclareMathOperator{\Char}{char}

\DeclareMathOperator{\Gal}{Gal}
\DeclareMathOperator{\Ind}{Ind}

\DeclareMathOperator{\Br}{Br}

\DeclareMathOperator{\Sym}{Sym}

\DeclareMathOperator{\Pic}{Pic}

\DeclareMathOperator{\Spec}{Spec}

\DeclareMathOperator{\et}{et}

\DeclareMathOperator{\N}{N}

\newcommand{\Am}[1]{k^{\times} \setminus k^{\times #1}} 
\newcommand{\hideqed}{\renewcommand{\qed}{}}

\newcommand{\isom}{\cong}

\numberwithin{equation}{section}
\numberwithin{table}{section}

\newcommand{\defi}[1]{\textsf{#1}} 

\title[Analogues of Chatelet surfaces]{
		Higher dimensional analogues of Ch\^atelet surfaces}
\subjclass[2000]{Primary 11 G35; Secondary 14 G05}

\author{Anthony V\'arilly-Alvarado}
\author{Bianca Viray}
\thanks{The second author was partially supported by NSF Grant DMS-0841321 and a Ford Foundation Dissertation Fellowship.  This collaboration was partially supported by Rice University.}

\address{Department of Mathematics MS 136, Rice University, Houston, TX 77005, USA}
\email{varilly@rice.edu}
\urladdr{http://www.math.rice.edu/\~{}av15}

\address{Department of Mathematics, Box 1917, Brown University, Providence, RI
			02912, USA}
\email{bviray@math.brown.edu}
\urladdr{http://math.brown.edu/\~{}bviray}

\date{}

\begin{document}

	\begin{abstract}
		We discuss the geometry and arithmetic of higher-dimensional analogues of Ch\^atelet surfaces; namely, we describe the structure of their Brauer and Picard groups and show that they can violate the Hasse principle.  In addition, we use these varieties to give straightforward generalizations of two recent results of Poonen.  Specifically, we prove that, assuming Schinzel's hypothesis,  the non-$m^{th}$ powers of a number field are diophantine.  Also, given a global field $k$ such that $\Char(k) = p$ or $k$ contains the $p^{th}$ roots of unity, we  construct a $(p+1)$-fold that has no $k$-points and no \'etale-Brauer obstruction to the Hasse principle.  
	\end{abstract}

	\maketitle

	%
	\section{Introduction}

		Our goal in this note is to draw attention to a particular class of smooth compactifications of varieties of the form
			\[
				\N_{K/k}(\vec{z}) = P(x),
			\]
		where $K/k$ is a finite extension of fields with associated norm form $\N_{K/k}$, and $P(x)$ is a polynomial in one variable of degree at least two.  The arithmetic of these projective varieties is studied, in essentially this level of generality, in \cites{Serre,CTS-Pencils,CTSSD-Schinzel,CT-Pest,CTHS}.  Theorems that apply to this broad class of varieties are relatively difficult to prove. However, one gets significant and strong results by imposing further hypotheses on $K/k$ and $P(x)$. On the one hand, one may require $K/k$ to be a quadratic extension of (say) number fields and $P(x)$ to be a separable polynomial (possibly of high degree), in which case we recover the study of conic bundle surfaces.  If further $P(x)$ is of degree $4$, then we recover Ch\^atelet surfaces, a class of varieties whose arithmetic is well understood \cites{Chatelet,CTS-descent,CTCS,CTSSDChatelet1,CTSSDChatelet2}.  On the other hand, one can let $K/k$ be more or less arbitrary and then restrict $P(x)$ to have a small number of roots \cites{CTS-Requivalence,HBS,CTHS,CT-Pest}.

		We specialize in a different direction that is ``orthogonal'' to the above cases. More precisely, we take $k$ a global field, $K/k$ a cyclic extension of prime degree $p$, and $P(x)$ a separable polynomial of degree $2p$.  These pencils of Severi-Brauer varieties coincide with the class of Ch\^atelet surfaces when $p = 2$;  for this reason we refer to these varieties as \defi{Ch\^atelet $p$-folds}.  

		We contend that Ch\^atelet $p$-folds are the proper high-dimensional generalizations, from both a geometric and an arithmetic perspective, of Ch\^atelet surfaces, and are therefore interesting objects to study.  In \S\ref{sec:ProperModels}, we give explicit smooth compactifications of their affine models.  In \S\ref{sec:PicardBrauer}, we give a complete description of their Picard and Brauer groups.  In \S\ref{sec:HassePrinciple}, we construct a Ch\^atelet $p$-fold over any global field $k$ with $\mu_p \subseteq k$ or $\Char(k) = p$ that violates the Hasse principle.  We list some further natural arithmetic questions in \S\ref{sec:Questions}.

		We remark that many ideas and computations involved in our proofs are standard; we record them for the reader's convenience.

		\subsection{Applications of Ch\^atelet $p$-folds}

			Our work suggests that, broadly speaking, any result whose proof uses Ch\^atelet surfaces can be generalized or re-proved using Ch\^atelet $p$-folds, provided one is willing to assume Schinzel's hypothesis~\cite{Schinzel}; see~\S\ref{subsec:diophantine_sets} for a statement of this conjecture over number fields. To illustrate this philosophy, we prove the following theorem on diophantine sets (recall that a subset $A \subseteq k$ is diophantine over $k$ if there is a closed subscheme $X \subseteq \Aff^{m+1}_k$ such that $A$ equals the projection of $X(k)$ under any map $k^{m+1} \to k$.)

			\begin{thm}
				\label{thm: pth powers}
				Assume Schinzel's hypothesis.  For any number field $k$ and any prime $p$, the set $k^{\times} \setminus k^{\times p}$ is diophantine over $k$.
			\end{thm}

			The following corollary follows almost immediately.

			\begin{cor}
				\label{cor: mth powers}
				Assume Schinzel's hypothesis.  For any number field $k$ and any natural number $m$, the set $k^{\times} \setminus k^{\times m}$ is diophantine over $k$.
			\end{cor}

			Poonen proved Theorem~\ref{thm: pth powers} for $p = 2$, unconditionally, using Ch\^atelet surfaces~\cite{Poonen-nonsquares}.  Our proof of Theorem~\ref{thm: pth powers} reduces to his; we require Schinzel's hypothesis to know that the Brauer-Manin obstruction is the only one for certain Ch\^atelet $p$-folds, namely those for which $P(x)$ splits as a product of two irreducible degree $p$ polynomials~\cite{CTSSD-Schinzel}*{Examples~1.6} (see also~\cite{CTS-Schinzel}); this result holds unconditionally when $p = 2$ by the landmark work of Colliot-Th\'el\`ene, Sansuc and Swinnerton-Dyer~\cite{CTSSDChatelet2}.  For an alternative proof of Theorem~\ref{thm: pth powers} in the case $p = 2$ and $k = \Q$, see~\cite{Koenigsmann}*{Proposition~17}.

			In a different direction, Poonen used Ch\^atelet surfaces in a ground-breaking paper to prove the insuffiency of the \'etale-Brauer set~\cite{Poonen-insufficiency}.  Specifically, for any global field $k$ of characteristic different from $2$, he constructed a $3$-fold $Y$ over $k$ that has no \'etale-Brauer obstruction and yet has no $k$-rational points.  Using Ch\^atelet $p$-folds, we obtain the following straightforward generalization of Poonen's construction.

			\begin{thm}
			\label{thm:insufficiency}
				Let $k$ be any global field such that either $k$ contains the $p^{th}$ roots of unity or $\Char(k) = p$. Then there exists a $(p+1)$-fold $Y$ such that $Y$ has no $k$-rational points and has non-empty \'etale-Brauer set.
			\end{thm}

			It is worth emphasizing that the significant ideas in the proofs of Theorems~\ref{thm: pth powers} and~\ref{thm:insufficiency} are due to Poonen~\cites{Poonen-nonsquares,Poonen-insufficiency}. We include these theorems only to support our claim that Ch\^atelet $p$-folds are interesting objects, amenable to explicit analysis. 

		\subsection{Notation}
			Throughout $p$ prime denotes a rational prime.  We write $\mu_p$ for the group of $p^{\textup{th}}$ roots of unity and $\zeta_p$ for a fixed a generator for $\mu_p$.  For a global field $k$ and any finite set $S$ of places of $k$, we write $\OO_{k,S}$ for the ring of $S$-integers.  If the characteristic of $k$ is zero, then $\OO_k$ will denote the usual ring of integers.  For a place $v$ of $k$, let $k_v$ denote the completion of $k$ at $v$, let $\OO_v$ denote the ring of integers of $k_v$, and let $\F_v$ denote the residue field. For a fixed separable closure $\ksep$ of $k$, we write $G_k$ for the absolute Galois group $\Gal(\ksep/k)$. For a $k$-scheme $X$, we write $\Xsep$ for $X\times_k \ksep$. 

		\subsection*{Acknowledgements}

			We thank Bjorn Poonen for suggesting we prove Theorem~\ref{thm: pth powers} and helpful conversations.  We also thank David Harari, Brendan Hassett and Alexei Skorobogatov for helpful discussions. We thank Tim Browning, Brian Conrad, Jean-Louis Colliot-Th\'el\`ene, Yongqi Liang and the anonymous referee for comments and corrections; we are particularly grateful to Colliot-Th\'el\`ene for the proof of Lemma~\ref{lem:CT}.

	\section{Smooth proper models of norm hypersurfaces}\label{sec:ProperModels}

		Let $K/k$ be a separable finite extension of degree $n$ and let $P(x) \in k[x]$ be a separable polynomial of degree $dn$, where $d$ is a positive integer.  Let $X_0$ be the affine \defi{norm hypersurface}
			\begin{equation}
			\label{eq:normic}
				\N_{K/k}(\vec{z}) = P(x)
			\end{equation}
			in $\Aff^{n + 1}_k$.  We construct a smooth proper model $X$ of $X_0$ that extends the map $X_0 \to \Aff^1_k$ given by $(\vec{z},x) \mapsto x$ to a map $X \to \PP^1$.

		Let $\calE$ be a rank $n + 1$ vector sheaf on $\PP^1_k$.  Given a section $s \in \Gamma(\PP^1, \Sym^n(\calE))$, we can construct the $n$-dimensional variety $V(s) \subseteq \PP\calE$.  By choosing a particular $\calE$ and $s$, we can ensure that $V(s)$ has an open affine isomorphic to $X_0$.

		Set $\calE = \left(\oplus_{i = 1}^n \OO\right)\oplus\calO(d)$, and let $s_2$ be the homogenization $\Ptilde(x,w) = w^{dn}P(x/w)$ in $\Gamma(\PP^1, \OO(d)^{\otimes n})$.  Let $s_1 = \N_{K/k}(\vec{z}) \in \Gamma(\PP^1, \Sym^n(\oplus_{i = 1}^n\OO))$.  Then $X_{K/k, P(x)} := V(s_1 - s_2) \in \PP\calE$ is a compactification of $X_0$.

		\begin{prop}
			The scheme $X = X_{K/k, P(x)}$ is smooth over $k$.
		\end{prop}
		\begin{proof}
			Since $\PP^1$ is smooth over $k$, $X$ is smooth at all points where $X\to \PP^1$ is smooth.  Thus it remains to show that the singular points in the fibers where $P(x) = 0$ are smooth points of $X$.  Since $P(x)$ is separable, this follows from the Jacobian criterion.
		\end{proof}

		We refer to $X$ as a \defi{degree $n$ normic bundle}. As mentioned earlier, we are mainly interested in Ch\^atelet $p$-folds, i.e.\ where $K/k$ is a cyclic extension of prime degree $p$ and $P(x)$ has degree $2p$.

	\section{The Picard and Brauer group of degree $p$ normic bundles}
	\label{sec:PicardBrauer}

		Let $K/k$ be cyclic of prime degree $p$, let $P(x)\in k[x]$ be a separable polynomial of degree $d$ divisible by $p$, and let $X := X_{K/k, P(x)}$ be a degree $p$ normic bundle.  

		\subsection{The Picard group}
		
			\begin{prop}\label{prop:Pic}
				The geometric Picard group of $X$ is of rank $d(p - 1) + 2$ and is freely generated by a smooth fiber, $(p-1)$ of the components of each degenerate fiber, and a section.  The group $\Pic \Xsep$ fits into the following (non-split) exact sequence of Galois modules.
				\[
					0 \to \frac{\oplus_{i = 1}^d\left(\Ind_K^k(\Z)\cdot \alpha_i\right)}{\langle \beta \cdot\alpha_i - \beta\cdot \alpha_j : \beta \in \left(\Ind_K^k(\Z)\right)^{G_k}\rangle}\to \Pic \Xsep \to \Z \to 0
				\]
				where $\alpha_i$ is a root of $P(x)$ for all $i$.
			\end{prop}
			\begin{proof}
				Let $(\Xsep)_{\eta}$ be the generic fiber of the map $\Xsep \to \PP^1$. We have the following exact sequence
				\begin{equation}\label{eqn:verticaldivisors}
				0 \to N \to \Pic \Xsep \to \Pic (\Xsep)_{\eta} \to 0,
				\end{equation}
				where $N$ is the subgroup generated by the vertical divisors.  Since $(\Xsep)_{\eta}$ is a Severi-Brauer variety, $\Pic (\Xsep)_{\eta} \isom \Z$, with trivial Galois action. The proposition follows because the degenerate fibers of $X$ lie over the roots of $P(x)$ and consist of the union of $p$-hyperplanes, all conjugate by elements of $\Gal(K/k)$.
			\end{proof}

		\subsection{The Brauer group of degree $p$ normic bundles}

			We use the cohomological description of the Brauer group, namely $\Br X := \HH^2_{\et}(X,\G_m)$. Note that since $X$ is a quasi-projective variety over a field, $\Br X$ is isomorphic to the group of Azumaya algebras over $X$ up to Morita equivalence~\cite{deJong}.

			Let $K/k$ be a finite cyclic extension of fields of degree $p$, and fix a generator $\sigma$ of $\Gal(K/k)$. Write $K[x]_\sigma$ for the ``twisted'' polynomial ring, where $\ell x = x{}^\sigma\!\ell$ for all $\ell \in K$.  If  $\Char k \neq p$, then given any $b \in k^{\times}$, we may construct the central simple $k$-algebra $K[x]_\sigma/(x^p - b)$; we denote this (cyclic) algebra by $(\chi_K,b)$, where $\chi_K\colon \Gal(K/k) \to \Z/p\Z$ is a character such that $\sigma \mapsto 1$.  Given two generators $\sigma$ and $\sigma'$ of $\Gal(K/k)$, we have $(\chi_K,b) = (\chi'_K,b)^m$, for $m$ such that $\sigma = \sigma'^m$.  If $\Char k = p$, then $(\chi_K,b)$ denotes the cyclic algebra $K[x]_\sigma/(x^p - x - b)$, with $\chi_K$ as above.

			\begin{theorem}\label{thm:Brauergroup}
				Let $P(x)$ be a separable polynomial and let $P_i(x)$ be the monic irreducible factors for $i = 1, \ldots, m$.  Let $c\in k^{\times}$ be such that $P(x) = c\prod P_i(x)$.  Denote $\deg P_i$ by $d_i$.  Let $K/k$ be a cyclic extension of prime degree and let $X := X_{K/k, P(x)}$.  If the splitting field of each $P_i(x)$ is not equal to $K/k$, then the map
				\[
					\frac{\left\{ (n_i)\in (\Z/p\Z)^m : \sum n_id_i \equiv 0 \pmod p\right\}}{(1, 1, \ldots, 1)}{\longrightarrow} \frac{\Br X}{\Br k} , \quad (n_i)_{i=1}^p \mapsto \left(\chi_K, \prod_i P_i(x)^{n_i}\right)_p
				\]
				is a group isomorphism. Otherwise, we may assume that the splitting field of $P_1(x)$ is equal to $K$.  In this case the above map is a surjection, and the kernel is generated by $e_1 := (1, 0, \ldots, 0)$.
			\end{theorem}
			\begin{proof}
				This argument is a generalization of~\cite{Skorobogatov-torsors}*{Props. 7.1.1 \& 7.1.2}.  We use the exact sequence of low-degree terms from the Hochschild-Serre spectral sequence
				\[
					0 \to \Pic X_{\eta} \to \Pic\left((X_{\eta})^{\textup{s}}\right)^{G_{k(x)}} 
					\to \Br k(x) \to 
					\ker\left(\Br X_{\eta} \to \Br (X_{\eta})^{\textup{s}}\right) 
					\to \HH^1(G_{k(x)},\Pic (X_{\eta})^{\textup{s}}).
				\]
				Since $(X_{\eta})^{\textup{s}}$ is rational, $ \Br (X_{\eta})^{\textup{s}} = 0$ and $\Pic (X_{\eta})^{\textup{s}} = \Z$ with a trivial Galois action. The quotient group $ \Pic\left((X_{\eta})^{\textup{s}}\right)^{G_{k(x)}}/\Pic X_{\eta}$ has order dividing $p$, and $(\chi_K, P(x))$ generates the order $p$ kernel $\Br k(x) \to \Br X_{\eta}$.  Thus we have
				\begin{equation}\label{eq:Braueriso}
					\Br X_{\eta} \isom  \frac{\Br k(x)}{\left(\chi_K, P(x)\right)}.
				\end{equation}
				We remark that this isomorphism also shows that $\Br \Xsep$ = 0, so $\Br X/\Br k \isom \HH^1(G_k, \Pic \Xsep)$.  We will return to this later.

				If $\sum n_id_i \equiv 0 \pmod p$, then, if $\Char(k) \neq p$, by the Purity theorem~\cite{Fujiwara}, $\calA := (\chi_K, \prod_i P_i(x)^{n_i})$ is everywhere unramified and hence is an element of $\Br X$.  If $\Char(k) = p$, then we must prove this ``by hand'', meaning that we first find open sets $U_i$ and elements $\calA_i \in \Br \kk(X)$, regular on $U_i$, such that $\calA = \calA_i$ in $\Br \kk(X)$ and then we use Lemma~\ref{lem:CT} below.  One can construct these representatives $\calA_i$ by using the the fact that $(\chi_K, P(x))$ is trivial in $\Br \kk(X)$; see~\cite{Poonen-chatelet, Viray-char2chatelet} for examples in the Ch\^atelet surface case.
				
				Combined with isomorphism~\eqref{eq:Braueriso}, the fact that these elements are unramified implies that the map $(n_i)_{i=1}^p \mapsto \left(\chi_K, \prod_i P_i(x)^{n_i}\right)_p$ is well-defined.  One can check that the kernel is generated by $\left(1, 1, \ldots, 1 \right)$ and, if $K/k$ is not disjoint from the splitting field of $P(x)$, by $e_1$.  It remains to prove that the map is surjective.

				Taking cohomology of~\eqref{eqn:verticaldivisors}, we obtain
				\[
					0 \to N^{G_k} \to (\Pic \Xsep)^{G_k} \to \Z 
					\to \HH^1(G_k, N) \to \HH^1(G_k, \Pic \Xsep) \to 0.
				\]

				Using the explicit generators given in Proposition~\ref{prop:Pic} and~\cite{Corn}*{Lemma~3.1}, one can check that
				\[
					0 \to N^{G_k} \to (\Pic \Xsep)^{G_k} \to \Z \to \Z/p\Z \to 0
				\]
				is exact and that $\HH^1(G_k, N)$ has the desired cardinality.  This shows that the map $(n_i)_{i=1}^p \mapsto \left(\chi_K, \prod_i P_i(x)^{n_i}\right)_p$ is surjective, completing the proof.
			\end{proof}

			Theorem~\ref{thm:Brauergroup} allows us to prove the following Corollary, which can also be deduced from~\cite{CTHS}*{Proposition~3.5}; see Remark~\ref{rem:CT} for further details.

			\begin{cor}
				If $P(x)$ is irreducible or the product of two irreducible polynomials, each of which has degree prime to $p$, then for any cyclic degree $p$ extension $K/k$, the Brauer group of $X_{K/k, P(x)}$ consists only of constant algebras.  
				In particular, there is no Brauer-Manin obstruction to the Hasse principle or weak approximation on $X_{K/k, P(x)}$. \qed
			\end{cor}
			
			The following lemma is used in the proof of Theorem~\ref{thm:Brauergroup} above and in Case 2 of Proposition~\ref{prop:HPcounterex} below in lieu of the Purity Theorem, when $\Char k = p$, to deal with the $p$-torsion of $X_{K/k,P(x)}$.  We are grateful to Jean-Louis Colliot-Th\'el\`ene for outlining the proof.
			
			\begin{lemma}
				\label{lem:CT}
				Let X be a regular, integral, quasi-compact scheme.  Suppose there is a Zariski open cover $\{U_i\}$ of $X$, together with elements $\alpha_i \in \Br U_i$ such that $\alpha_i = \alpha_j$ (all $i$ and $j$) when viewed as elements of $\Br \kk(X)$ under the natural inclusion $\Br U_i \to \Br \kk(X)$. Then there is an element $\alpha \in \Br X$ that restricts to $\alpha_i$ under the natural inclusion $\Br X \to \Br U_i$ for all $i$.
			\end{lemma}
			
			\begin{proof}
			Let $i\colon \Spec \kk(X) \to X$ be the inclusion of the generic point, and let $X^{(1)}$ be the set of closed integral subschemes of $X$ of codimension 1. For $x \in X^{(1)}$, write $i_x\colon \Spec \kappa(x) \to X$ for the inclusion of the generic point.   The short exact sequence of \'etale sheaves on $X$
			\[
				0 \to \G_{m,X} \to i_*\G_{m,\kk(X)} \to \bigoplus_{x \in X^{(1)}} i_{x*}\Z \to 0
			\]
			gives rise to a long exact sequence in cohomology
			\begin{equation}
				\label{eq:LES in Etale Cohomology}
				0 \to \HH^2_{\et}(X,\G_m) \to \HH^2_{\et}(X,i_*\G_{m,\kk(X)}) \to \bigoplus_{x \in X^{(1)}}\HH^1(\kappa(x),\Q/\Z);
			\end{equation}
			see~\cite{Milne-EtaleCohomology}*{Chap.~III, Example 2.22}. Consider the commutative diagram
			\[
				\xymatrix{
					 & & & 0\ar[d] \\
					0 \ar[r] & \HH^2_{\et}(X,i_*\G_{m,\kk(X)}) \ar[r] \ar[d] & \Br \kk(X) \ar[d]^\Delta \ar[r] & \HH^0_{\et}(X,R^2i_*\G_{m,\kk(X)}) \ar[d] \\
					0 \ar[r] & \prod_i\HH^2_{\et}(U_i,i_*\G_{m,\kk(X)}) \ar[r] & \prod_i \Br \kk(U_i) \ar[r] & \prod_i \HH^0_{\et}(U_i,R^2i_*\G_{m,\kk(X)})
				}
			\]
			Here the top row comes from the low-degree exact sequence associated to the Leray spectral sequence for the morphism $i\colon \Spec \kk(x) \to X$, and we use the equality $R^1i_*\G_{m,\kk(X)} = 0$, which follows from Hilbert's Theorem 90. The bottom row is obtained similarly using the Zariski cover $\{U_i\}$ of $X$, and the map $\Delta$ is the diagonal embedding. Exactness of the last column is just part of the (\'etale) sheaf axiom for $R^2i_*\G_{m,\kk(X)}$. Using the analog of the sequence~\eqref{eq:LES in Etale Cohomology} for each $U_i$, we see that the $\alpha_i \in \Br U_i$ together give an element of $\prod_i\HH^2_{\et}(U_i,i_*\G_{m,\kk(X)})$ whose image in $\prod_i \Br \kk(U_i)$ is contained in the image of the map $\Delta$. A diagram chase shows that this element can be lifted to $\HH^2_{\et}(X,i_*\G_{m,\kk(X)})$, and by hypothesis, this element is in the kernel of the residue map in~\eqref{eq:LES in Etale Cohomology}. The lemma follows immediately.
			\end{proof}

		\begin{remark}
			\label{rem:CT}
			In~\cite{CTHS}, Colliot-Th\'el\`ene, Harari, and Skorobogatov give smooth partial compactifications of the affine varieties~\eqref{eq:normic} and study the Picard and Brauer groups of these partial compactifications.  Their results, applied to our compactification, already show that there is an inclusion of the Brauer group
			\[
				\frac{\Br X}{\Br k} \hookrightarrow \frac{\left\{ (n_i)\in (\Z/p\Z)^m \right\}}{(1, 1, \ldots, 1)}.
			\]
			See in particular~\cite{CTHS}*{Proposition~2.5}.
		\end{remark}

	\section{A Ch\^atelet $p$-fold that violates the Hasse principle}
	\label{sec:HassePrinciple}
		\begin{prop}\label{prop:HPcounterex}
			Fix a prime $p$.  Let $k$ be any global field such that either $\mu_p \subseteq k$ or $\Char(k) = p$.  Then there exists a Ch\^atelet $p$-fold $X$ over $k$ that violates the Hasse principle.
		\end{prop}

		\begin{proof}
			This is a straightforward generalization of a construction of Poonen~\cite{Poonen-chatelet}*{\S\S5\&11} if $\mu_p \subseteq k$, and of~\cite{Viray-char2chatelet} if $\Char(k) = p$. We outline the specific necessary modifications for the reader's convenience.
			
			\textbf{Case 1a: $\mu_p \subseteq k$, $\Char(k) = 0$}

			Let $N$ be such that for any finite field $\F$ of cardinality greater than $N$, any smooth degree $p$ plane curve over $\F$ has at least $3p + 1$ $\F$-points.  By the Chebotarev density theorem and global class field theory, we can find $b\in \OO_k$, generating a prime ideal, such that $b \equiv 1 \pmod{ \left(1 - \zeta_p\right)^{2p-1}\OO_k}$ and $\#\F_b > N$.  Similarly, we can find $a\in \OO_k$, generating a prime ideal, such that $a \equiv 1\pmod{ \left(1 - \zeta_p\right)^{2p-1}\OO_k}, a \notin k_b^{\times p}$ and $\#\F_a > N$.  Let $c\in\OO_k$ be such that $b\mid\left(ac + 1\right)$ and let $K$ be the degree $p$ Kummer extension $k\left( \sqrt[p]{ab}\right)$. Let $X$ be the smooth projective model of
			\[
				\N_{K/k}\left(\vec{z}\right) = 
					\left(x^p + c\right)\left(ax^p + ac + 1\right).
			\]
			We claim that $X$ has a Brauer-Manin obstruction to the Hasse principle, given by the cyclic algebra $(\chi_K, x^p + c)$.  The proof proceeds exactly as in~\cite{Poonen-chatelet}*{\S5}: for the existence of local points, c.f.~\cite{Poonen-chatelet}*{Lemma 5.3} and for the Brauer-Manin obstruction, c.f.~\cite{Poonen-chatelet}*{Lemma 5.5}.

			\textbf{Case 1b: $\mu_p \subseteq k$, $\Char(k) \neq 0, p$}

				Fix a prime $\pp$, and let $\OO_k = \OO_{k, \{\pp\}}$.  Then use $a$, $b$ as in Case (1a), replacing the condition that $a,b \equiv 1 \pmod{ \left(1 - \zeta_p\right)^{2p-1}\OO_k}$ with the condition that $a,b$ are $p^{th}$ powers in $k_{\pp}$.  Then same construction as in Case (1a) gives a Ch\^atelet $p$-fold with a Brauer-Manin obstruction to the Hasse principle and the obstruction is caused by the same element. As in Case (1a), one uses~\cite{Poonen-chatelet}*{Lemmas 5.3 and 5.5} for the existence of local points and the computation of the Brauer-Manin obstruction, respectively.

			\textbf{Case 2: $\Char(k) = p$.}

				Let $\mathbb F$ denote the constant field of $k$ and let $n$ denote the order of $\mathbb F^{\times}$.  Fix a prime $\pp$ of $k$ of degree prime to $p$ and let $S = \{\pp\}$.  Let $\gamma \in \mathbb{F}$ be such that $T^p - T + \gamma$ is irreducible in $\mathbb{F}[T]$.    By the Chebotarev density theorem we can find elements $a,b \in \mathcal{O}_{k,S}$ that generate prime ideals of degree divisible by $p$ and degree prime to $p$, respectively, such that $a \equiv \gamma \pmod{b^2\OO_{k,S}}$.  These conditions imply that $v_\pp(a)$ is equivalent to $0\pmod p$ and negative and that $v_{\pp}(b)$ is prime to $p$ and negative.

				Define
				\begin{eqnarray*}
					f(x) & = & \left(a^{-4n} b\right)^{p-1} x^p - x - a b^{-1},\\
					g(x) & = & a^{-4np} b^p x^p - a^{-4n} b x - a^{1-4n} + \gamma.
				\end{eqnarray*}
				Note that $g(x) = a^{-4n} b f(x) + \gamma.$  
				Let $K = k[T]/(T^p - T + \gamma)$, and let $X$ be the Ch\^atelet surface given by
				\begin{equation*}\tag{$*$}\label{eq:chatelet}
					\N_{K/k}(\vec{z}) = f(x)g(x).
				\end{equation*}
				The cyclic algebras $\calA := (\chi_K, f(x))$, $(\chi_K, g(x))$ and $(\chi_K, f(x)/x^p)$ all represent the same class in $\Br \kk(X)$, so by Lemma~\ref{lem:CT}, we have $\calA \in \Br X$.	We claim that $X$ has a Brauer-Manin obstruction to the Hasse principle given by the element $\calA$.  The proof proceeds exactly as in~\cite{Viray-char2chatelet} (for the existence of local points, c.f.~\cite{Viray-char2chatelet}*{Lemma~3.1} and for the Brauer-Manin obstruction, c.f~\cite{Viray-char2chatelet}*{Lemma~3.3}), with one exception.  To show that $X(k_{\pp})$ is non-empty, one shows that $f(x)g(x)$ has valuation divisible by $p$ at $x = ab^{-1}\pi^{-m}$, where $\pi$ is a uniformizer for $\pp$ and $m$ is the smallest positive integer that is congruent to $-v_{\pp}(b) \mod p$. In addition, for the analog of \cite{Viray-char2chatelet}*{Lemma~3.3}, one uses that $T^p  - T - \gamma$ is irreducible in $\F[T]$ to show that the local invariant at $b$ is constant and nonzero, and that all other local invariants are trivial.
		\end{proof}

	\section{Proofs of Theorems~\ref{thm: pth powers}
	and~\ref{thm:insufficiency}}
	\label{sec:pfs}

		\subsection{Diophantine Sets}
		\label{subsec:diophantine_sets}
			Our goal in this subsection is to prove Theorem~\ref{thm: pth powers}.  We first need to generalize a few results from~\cite{Poonen-nonsquares}.  Throughout this section $k$ is a number field.

			\begin{lemma}\label{surjection}
				Fix $P(x)\in k[x]$ a separable polynomial of degree divisible by $p$, and let $\alpha \in k^\times$.  Let $K$ denote the cyclic extension $k(\sqrt[p]{\alpha})$ and let $X := X_{K/k, P(x)}$.  Fix a non-trivial character $\chi_K$ of $\Gal(K/k)$.  Then there is a finite set of places $S$, depending on $P(x)$ but not $\alpha$, such that if $v\notin S$, $\Br X \neq \Br k$, and $v(\alpha) \neq 0 \pmod p$, then
				\[
					X(k_v) \to \left(\Br k_v\right)[p]
				\]
				is surjective.
			\end{lemma}
			\begin{proof}
				Let $Q(x)$ be such that $\calA := (\chi_{K}, Q(x))$ is a non-trivial element in $\Br X/\Br k$.  (All elements of $\Br X/\Br k$ are of this form by Theorem~\ref{thm:Brauergroup}.)  Then there exists a finite set of places $S$, depending only on $P(x)$ such that for all $v\notin S$, the variety $Y$
				\begin{equation}\label{eqn:ancillaryvariety}
					z^p = P(x), \quad y^p = \beta^{p-1}Q(x)
				\end{equation}
				has a smooth $\F_v$-point (and hence a $k_v$-point), where $\beta$ is any $v$-adic unit.

				Let $S$ be as above, and take $v \notin S$, $\alpha \in k^{\times}$ such that $v(\alpha) \not \equiv 0 \pmod p$.  Take $\calB \in (\Br k_v)[p]$ and choose $\beta \in k_v^{\times}$ such that $\calB = (\chi_K, \beta)$.  Let $(x,y,z)$ be a $k_v$-point of $Y$.  One can check that $Y$ maps to $X$, sending $(x,y,z)$ to $((z,0 \ldots,0),x)$.  The equations of $Y$ imply that $\calA((z,0 \ldots,0),x) = \calB$, completing the proof.
			\end{proof}

			Proceeding as in~\cite{Poonen-nonsquares}*{\S3}, using Lemma~\ref{surjection} to replace~\cite{Poonen-nonsquares}*{Lemma~3.1}, we obtain the following natural generalization of~\cite{Poonen-nonsquares}*{Theorem~1.3}.

			\begin{theorem}\label{thm: finite classes}
				Let $P(x) \in k[x]$ be a nonconstant separable polynomial of degree divisible by $p$. Then there are at most finitely many classes in $k^\times/k^{\times p}$ represented by $\alpha \in k^\times$ such that there is a Brauer-Manin obstruction to the Hasse principle for $X_{\alpha,P(x)}$. \qed
			\end{theorem}

			\begin{lemma}
				\label{lem:reduction}
				Let $K = k(\zeta_p)$. Assume that $K^\times\setminus K^{\times p}$ is diophantine over $K$. Then $\Am{p}$ is diophantine over $k$.
			\end{lemma}

			\begin{proof}
				By~\cite{Shlapentokh}*{\S2.2}, the set $(K^\times\setminus K^{\times p})\cap k$ is diophantine over $k$. Since $p$ is odd, by~\cite[Theorem 5.6.2]{Cohen} we have $K^{\times p}\cap k = k^{\times p}$; thus $(K^\times \setminus K^{\times p})\cap k = \Am{p}$.
			\end{proof}
			
			We are now in a position to prove Theorem~\ref{thm: pth powers}. For the reader's convenience, we first state Schinzel's hypothesis (c.f.~\cite{CTS-Pencils}*{\S4}).
			
			\medskip
			
			\noindent{\bf Schinzel's hypothesis for number fields:} Let $k$ be a number field, and let $P_i(t) \in k[t]$  be irreducible polynomials for $i = 1, \ldots, n$.  Let $S$ be a finite set of places of $k$ containing all archimedean places, all finite places $v$ where $P_i(t) \not \in \OO_v[t]$ for some $i$, and all finite places above a prime $p$ that is less than or equal to the degree of $\N_{k/\Q}\left(\prod_i P_i(t)\right)$.  

Given elements $\lambda_v \in k_v$ for all $v \in S$, there exists a $\lambda\in k$ such that
\begin{enumerate}
    \item $v(\lambda)\geq 0$ for all $v \notin S$,
    \item $\lambda$ is arbitrarily close to $\lambda_v$ for all finite $v \in S$,
    \item $\lambda$ is arbitrarily large in every archimedean completion.
    \item For all $i$, there is at most one place $v_i \notin S$ such that $P_i(\lambda) \notin k_v^{\times}$, and at $v_i$, $P_i(\lambda)$ is a uniformizing parameter.
\end{enumerate}

			\begin{proof}[Proof of Theorem~\ref{thm: pth powers}]
				By Lemma~\ref{lem:reduction}, we may assume without loss of generality that $\mu_p \subseteq k$. By Case 1(a) of the proof of Proposition~\ref{prop:HPcounterex}, there is a Kummer extension $K := k\left(\sqrt[p]{ab}\right)$ of $k$ and a Ch\^atelet $p$-fold $X_1$ associated to the affine variety
				\[
					\N_{K/k}\left(\vec{z}\right) = 
					\left(x^p + c\right)\left(ax^p + ac + 1\right)
				\]
				that fails to satisfy the Hasse principle on account of a Brauer-Manin obstruction. For $t \in k^\times$, let $X_t$ be the Ch\^atelet $p$-fold associated to the affine variety
				\[
					U_t :\ \N_{k\left(\sqrt[p]{tab}\right)/k}\left(\vec{z}\right) = 
					\left(x^p + c\right)\left(ax^p + ac + 1\right),
				\]				
				and let $f_t\colon U_t \to \Aff^1_k$ be the map $(\vec{z},x) \mapsto x$.
				We claim that the following statements are equivalent:
				\begin{enumerate}
					\item[(i)] $U_t$ has a $k$-point.
					\item[(ii)] $X_t$ has a $k$-point.
					\item[(iii)] $X_t$ is locally soluble and there is no Brauer-Manin obstruction to the Hasse principle for $X_t$.
				\end{enumerate}
				The implications (i) $\implies$ (ii) $\implies$ (iii) are clear.  Finally, if (iii) holds, then by~\cite{CTSSD-Schinzel}*{Theorem~1.1(b) and Examples~1.6} (assuming Schinzel's hypothesis), there is a smooth fiber of $f$ over a $k$-point $P \in \Aff^1_k$ such that $f^{-1}_t(P)(\Adeles_k) \neq \emptyset$. Since $f^{-1}_t(P)$ satisfies the Hasse principle (because $K/k$ is cyclic), we have $f^{-1}_t(P)(k) \neq \emptyset$, so $f^{-1}_t(P)$ is isomorphic to a projective space, and therefore (i) holds.

				The rest of the proof follows almost verbatim that of~\cite{Poonen-nonsquares}*{Theorem~1.1}, replacing~\cite{Poonen-nonsquares}*{Theorem~1.3} with Theorem~\ref{thm: finite classes}.
			\end{proof}

			\begin{proof}[Proof of Corollary~\ref{cor: mth powers}]
				For any integer $m \geq 2$, let $A_m = \Am{m}$.  By Theorem~\ref{thm: pth powers}, for any prime $p$, the set $A_p$ is diophantine over $k$.  We show first that $A_{p^n}$ is diophantine for any rational prime $p$ and $n\geq 1$ by induction on $n$.  The case $n = 1$ is true by hypothesis. Now note that $A_{p^{n+1}} = A_{p} \cup \{\, t^p : t \in A_{p^n} \,\}$, and that the union of finitely many diophantine sets is diophantine. Let $m = p_1^{n_1}\dots p_r^{n_r}$ be the prime factorization of $m$ (over $\Z$). The corollary then follows from the equality
				\[
					A_m = \bigcup_{i = 1}^r A_{p_i^{n_i}}. \tag*{\qed}
				\]
				\hideqed
			\end{proof}

		\subsection{Insufficiency of the \'etale-Brauer obstruction}

			\begin{proof}[Proof of Theorem~\ref{thm:insufficiency}]
				Most of~\cite{Poonen-insufficiency} goes through exactly as stated after replacing~\cite{Poonen-chatelet}*{Prop. 5.1, \S11} with Proposition~\ref{prop:HPcounterex}.  The one exception is \S5 of~\cite{Poonen-insufficiency}.  The motivated reader can determine for themselves how to generalize this argument to Ch\^atelet $p$-folds (the argument does generalize, but perhaps not obviously).  Alternatively, we can use~\cite{CT-0cycles}*{Prop. 2.1} instead, which does  generalize in a straight-forward way.  In \S6 of \cite{Poonen-insufficiency} $s_1$ is replaced with $u^p \widetilde{P}_{\infty}(w,x) + v^p \widetilde{P}_0(w,x)$, or, in characteristic $p$, with $u^p \widetilde{P}_{\infty}(w,x) + u^{p-1}v\widetilde{P}_{\infty}(w,x) +  v^p \widetilde{P}_0(w,x)$
			\end{proof}

			\begin{remark}
				Colliot-Th\'el\`ene proved that, although the $3$-folds Poonen constructed have no rational points, they do have $0$-cycles of degree $1$~\cite{CT-0cycles}.  The same result holds for the $(p+1)$-folds in Theorem~\ref{thm:insufficiency}; it follows from either~\cite{Liang-0cycles}*{Th\'eor\`eme principal} or~\cite{Wittenberg-0cycles}*{Th\'eor\`eme~1.3}. 
			\end{remark}

	\section{Open Problems}\label{sec:Questions}

		Colliot-Th\'el\`ene, Sansuc and Swinnerton-Dyer proved unconditionally that for Ch\^atelet surfaces the Brauer-Manin obstruction to the Hasse principle and weak approximation is the only one~\cite{CTSSDChatelet1, CTSSDChatelet2}.  Colliot-Th\'el\`ene, Skorobogatov, and Swinnerton-Dyer proved the same results for Ch\^atelet $p$-folds (and other varieties), assuming Schinzel's hypothesis~\cite{CTSSD-Schinzel}.  This suggests the following problem.

		\begin{problem}
			Prove \emph{unconditionally} that the Brauer-Manin obstruction is the only obstruction to the Hasse principle and weak approximation for Ch\^atelet $p$-folds.
		\end{problem}

		The difficulty of applying known methods to prove the sufficiency of the Brauer-Manin obstruction (e.g.\ universal torsors and the descent machinery~\cite{CTS-descent2}) seems to grow with $p$.  This suggests that a solution to this problem will require some new insight.

		In a more geometric direction, Colliot-Th\'el\`ene, Coray, and Sansuc proved that certain types of Ch\^atelet surfaces are $k$-unirational if there is a $k$-point~\cite{CTCS} when $k$ is a number field.  In~\cite{CTSSDChatelet2}*{Proposition 8.3}, this result is generalized to arbitrary Ch\^atelet surfaces over fields of characteristic $0$.  (We note that for arbitrary conic bundles, proving $k$-unirationality is quite difficult; for some special cases see~\cites{Mestre1, Mestre2, Mestre3}.)  Additionally, work of Koll\'ar shows that the same statement holds for Ch\^atelet $p$-folds when $k$ is a local field of characteristic $0$~\cite{Kollar-unirationality}*{Corollary 1.8}.  This leads us to the following natural question.
		\begin{problem}
			Let $k$ be a number field.  Does the existence of a $k$-point imply $k$-unirationality for Ch\^atelet $p$-folds, perhaps under some assumptions on the factorization of $P(x)$?
		\end{problem}


\end{document}